\RequirePackage{fix-cm}
\documentclass[twocolumn]{svjour3}          

\smartqed  
\usepackage[T2A]{fontenc}
\usepackage[utf8]{inputenc}
\usepackage[english]{babel}
\usepackage{graphicx}
\usepackage{multicol}
 \usepackage{float}
\usepackage{subcaption}
\usepackage{amssymb, amsmath}
\usepackage{amsthm}
\usepackage[misc]{ifsym}
\usepackage{color}
\usepackage{array}
\usepackage{tabu}
\usepackage{enumitem}
\usepackage{mathrsfs}

\begin{document}

\title{Precession of the Kovalevskaya and Goryachev-Chaplygin tops
}
\subtitle{}


\author{Ivan Polekhin}


\institute{Ivan Polekhin (\Letter) \at
            Steklov Mathematical Institute of Russian Academy of Sciences,\\
              Russia, Moscow, Gubkina str. 8, 119991 \\
              \email{ivanpolekhin@mi.ras.ru}           
}

\date{Received: date / Accepted: date}

\maketitle

\begin{abstract}
The change of the precession angle is studied analytically and numerically for the integrable tops of Kovalevskaya and Goryachev-Chaplygin. Based on the known results on the topology of Liouville foliations for these systems, we find initial conditions for which the average change of the precession angle is zero or can be estimated asymptotically. Some more difficult cases are studied numerically.
\end{abstract}

\section{Introduction}
Let us consider a Liouville integrable Hamiltonian system and suppose that the level sets of the first integrals are compact. The motion in such a system is always a periodic or quasiperiodic winding of the invariant torus. In special action-angle variables, the equations of motion have the following simple form:
$$
\dot I = 0, \quad \dot \varphi = f(I).
$$
These equations are in some sense convenient since their solutions can be presented explicitly in coordinates $I$, $\varphi$. At the same time, the simple form of the system and its integrability do not directly lead to the understanding of dynamics in original variables that have clear mechanical or geometrical interpretation. Many classical mechanical systems, especially integrable tops, can be considered as examples of such situations.

For instance, let us have a rigid body with a fixed point in a gravitational field and suppose that this system is the Kovalevskaya top. Though this system is integrable, even the qualitative picture of its motion in the absolute space is not so simple and it is a relatively hard problem to describe the dynamics of the top. The equations of motion can be presented as follows (see, for instance, \cite{golub53,kozlov2000,bolsinov2004integrable,whittaker1970treatise})
\begin{equation}
\label{eq1}
    \begin{aligned}
        &2\dot p - qr = 0, \quad 2\dot q + rp = \mu\gamma_3, \quad \dot r = -\mu\gamma_2,\\
        &\dot\gamma_1 = r\gamma_2 - q\gamma_3,\quad \dot\gamma_2 = p\gamma_3 - r\gamma_1, \quad \dot\gamma_3 = q\gamma_1 - p\gamma_2.
    \end{aligned}
\end{equation}
Here $p,q,r,\gamma_1,\gamma_2,\gamma_3 \in \mathbb{R}$ and $\mu \in \mathbb{R}$ is a parameter. Variables $p$, $q$, $r$ are the projections of the angular velocity vector on the principal axes of inertia, $\gamma_1$, $\gamma_2$, $\gamma_3$ are the projections of the unit vertical vector on the same axes. The first integrals have the form
\begin{equation}
    \begin{aligned}
        &p^2 + q^2 + \frac{1}{2}r^2 + \mu\gamma_1 = h, \quad 2(p\gamma_1 + q\gamma_2) + r\gamma_3 = c,\\
        &\gamma_1^2 + \gamma_2^2 + \gamma_3^2 = 1, \quad (p^2 - q^2 - \mu\gamma_1)^2 + (2pq - \mu\gamma_2)^2 = k^2.
    \end{aligned}
\end{equation}

In 1896 N.\,E.\,Joukowski offered \cite{zhuk1948} an interpretation of solutions of system (\ref{eq1}) that was similar to the Poinsot's interpretation of motion of the Euler top. However, in contrast to the latter, the surface that is rolling on a plane in the interpretation of Joukowski, is not a closed surface but have a complex form with self-intersections, i.e. it cannot be embedded in $\mathbb{R}^3$.

Note that various interpretations of motion can be useful when we need an in-between view on the dynamics of the system that is less complex than the `explicit' quadratures and more detailed than the general statement of the Liouville-Arnold theorem.

One of the possible approaches to the description of motion of the Kovalevskaya top in the absolute space is provided by the following result proved in \cite{kozlov2000} (also see \cite{kozlov2013behaviour})

\begin{theorem}
Let us consider a solution of (\ref{eq1}) and suppose that this solution never passes through the point where $\gamma_3 = \pm 1$, i.e. the standard Euler angles are correctly defined along this solution. Also suppose that the functions of first integrals are independent on the considered level set of the first integrals. Then the line of nodes of the system has a mean motion $\Lambda$.
\end{theorem}
\begin{remark}
In \cite{kozlov2000} this result was proved for non-degenerate invariant tori. Later, in \cite{kozlov2013behaviour} the theorem was proved without this assumption.
\end{remark}
\begin{remark}
Here the value $\Lambda$ is the same for a given invariant torus, yet can be different for various tori.
\end{remark}
\begin{remark}
The notion of a mean motion goes back to Celestial mechanics \cite{charlier1902mechanik,weyl1938mean,weyl1939mean2}.
\end{remark}
To be more precise, the result means that the change of the precession angle $\psi$ as a function of time has the form
\begin{equation}
\label{eq3}
  \psi(t) = \psi_0 + \Lambda t + f(\varphi_1^0 + \omega_1 t, \varphi_2^0 + \omega_2 t) - f(\varphi_1^0, \varphi_2^0),
\end{equation}
where $f$ is a continuous function on a two-dimensional torus.

The system (\ref{eq1}) has three degrees of freedom and $\psi$ can be considered as a cyclic variable. After the reduction w.r.t. $\psi$, we have an integrable system with two degrees of freedom. Therefore, $\varphi_1$ and $\varphi_2$ are the angular variables on the invariant torus and $\omega_1$, $\omega_2$ are the corresponding frequencies. Function (\ref{eq3}) defines a mean motion of the line of nodes (even when the frequencies are rationally dependent).

The change of the angles $\varphi_1$ and $\varphi_2$ does not depends on $\psi$. Therefore, from the theorem, we obtain that the motion of the radius-vector of the axis of dynamical symmetry in the absolute space is a composition of two motions. First, if we put $\Lambda = 0$, then the radius-vector moves on the unit sphere. If $\Lambda \ne 0$ then the final motion is the composition of the motion on the sphere and the rotation around the vertical axis with the angular velocity $\Lambda$. This interpretation of motion is close to the classical picture of motion in the Lagrange case.

In a typical case, when $\omega_1$ and $\omega_2$ are rationally independent, the trajectory of a solution is everywhere dense on the invariant torus. Again, suppose that our solution never passes through the points $\theta = 0$ and $\theta = \pi$. Then the angle $\psi$ is a continuous function of $\varphi_1$ and $\varphi_2$, i.e. $\theta = \theta(\varphi_1, \varphi_2)$. Therefore, in a typical situation, when the trajectory is everywhere dense, the unit vector parallel to the symmetry axis covers some region $D$. At the same time, this region is rotating with the constant angular velocity $\Lambda$.

Note that the existence of a mean motion of the precession angle (or any other function of the phase variables) is not obvious. For instance, for the Goryachev-Chaplygin top, the precession angle do not has a mean motion even in the cases when solutions are separated from the positions where $\theta = 0$ и $\theta = \pi$. However, it has a so-called main motion \cite{kozlov2013behaviour}. The definition of a main motion will be given below in the next section.

Taking into account the result on the existence of a mean motion in the Kovalevskaya case, it is natural to try to find the dependence of $\Lambda$ on the initial data. For instance, we can try to find the initial data for which a mean motion of the precession angle is zero. In \cite{kozlov2000}, the following was proved
\begin{theorem}
Let $c = 0$, then for $\mu$ small in absolute value we have $\Lambda = 0$.
\end{theorem}
\begin{remark}
Note that a mean motion is zero even for rationally dependent frequencies.
\end{remark}
Similar result was proved in \cite{kozlov2000} for a main motion of the line of nodes for the Goryachev-Chaplygin case (also for small $\mu$). Below we also consider the Kovalevskaya and Goryachev-Chaplygin tops and generalize results from \cite{kozlov2000}. The main aim of the paper is to study $\Lambda$ as a function of the initial data (for a mean and a main motion). When we prove that $\Lambda = 0$, we use known results on the topology of the Liouville foliation for the considered systems. In other cases, we study $\Lambda$ numerically.
\section{Auxiliary results and definitions}
\noindent Let us now define what we call a mean and a main motion.
\begin{definition}
We say that a dynamical variable (a function of time) $\psi(t)$ has a mean motion $\Lambda$ if for all $t$ we have $\psi(t) = \psi_0 + \Lambda t + O(1)$, i.e. it can be presented as a sum of a bounded function and a linear function of time.
\end{definition}

\begin{definition}
We say that a dynamical variable (a function of time) $\psi(t)$ has a main motion $\Lambda$ if for $t \to +\infty$ we have $\psi(t) = \psi_0 + \Lambda t + o(t)$, i.e. there exists a limit
$$
\lim\limits_{t \to +\infty} \frac{\psi(t)}{t} = \Lambda.
$$
\end{definition}
\noindent Let us consider an integrable system, defined locally by the system
\begin{equation}
\label{int_sys}
\dot y_1 = ... =\dot y_s = 0, \quad \dot x_1 = \omega_1(y) \,...\,\, \dot x_k = \omega_k(y),
\end{equation}
where $x_i$ are $2\pi$-periodic angle variables, i.e. we suppose that locally the phase space is foliated by tori and diffeomorphic to $D \times \mathbb{T}^k$, where $D \subset \mathbb{R}^s$ is a disk. Let $\psi$ be an angular variable on a torus $\mathbb{T}^k$, i.e. it is a some multivalued function that changes by $2\pi n$ (for some $n \in \mathbb{Z}$) along any closed path on the torus.
\begin{theorem}
The change of the angular variable $\psi(t)$ along a solution of (\ref{int_sys}) has the following form
$$
\psi(t) = t \cdot \sum\limits_{i=1}^k m_i \omega_i(y_0) + S(\omega(y_0)t + x_0, y_0) - S(x_0,y_0).
$$
Here $m_i \in \mathbb{Z}$, $S$ is a continuous function which is $2\pi$-periodic in angular variables, $x_0$ and $y_0$ are the initial data.
\end{theorem}
\noindent Let us now have an integrable Hamiltonian system with a cyclic variable, i.e. its Hamiltonian has the form
$$
H = H(q_1,...,q_n,p_1,...,p_n,J),
$$
where $J$ is a first integral corresponding to the cyclic variable. Let us denote the cyclic variable by $\psi$. From Theorem 3, we have
\begin{corollary}
The change of the cyclic variable along a solution of an integrable Hamiltonian system has the form $\psi(t) = \Lambda t + s(t)$, where $s$ is a quasi periodic function. Moreover, $\Lambda$ is a continuous function of the constants of first integrals, yet it does not depend on the initial data on a given invariant torus.
\end{corollary}

For the sake of completeness, we also present some standard definitions and results from ergodic theory.

\begin{definition}
Let $M$ be a smooth manifold, $\mu$ be a measure with a continuous positive density on $M$, $\varphi_t$ be a one-parameter group of measure-preserving diffeomorphisms:
$$
\mu(A) = \mu(\varphi_t(A))\mbox{ for any measurable set }A.
$$
We will call the triple $(M, \mu, \varphi_t)$ a dynamical system.
\end{definition}
\begin{definition}
Let $(M, \mu, \varphi_t)$ be a dynamical system, $0 < \mu(M) < \infty$ and $f \colon M \to \mathbb{R}$ be a $\mu$-measurable function. We will call $\bar f$ the space average of $f$
$$
\bar f = \frac{1}{\mu(M)}\int\limits_{M} f \, d\mu.
$$
$f^*(x)$ is the time average of $f$ (if exists)
$$
f^*(x) = \lim\limits_{T \to + \infty} \frac{1}{T} \int_0^T f(\varphi_t(x))\, dt.
$$
\end{definition}
\begin{definition}
A dynamical system $(M, \mu, \varphi_t)$ is ergodic if for any $\mu$-summable function $f$ we have $f^*(x) = \bar f$ a.e.
\end{definition}

One of the main results of ergodic theory is the Birkhoff-Khinchin theorem \cite{cornfeld2012ergodic}

\begin{theorem}
For almost all (w.r.t. $\mu$) $x \in M$ there exists the time average $f^*(x)$. Moreover, $f^*$ is a $\mu$-measurable function and
$$
\int\limits_M f d\mu = \int\limits_M f^* d\mu.
$$
\end{theorem}
\begin{remark}
Note that the ergodicity is not assumed in the statement of the Birkhoff-Khinchin theorem.
\end{remark}

\section{Rigid body with a fixed point: the general case}
Before proceeding to the consideration of integrable cases, we show how the ergodic Birkhoff-Khinchin theorem can be applied to a qualitative study of the main motion in the general (nonintegrable) case of motion of a rigid body with a fixed point in a gravity field. Let $A, B, C > 0$ be the moments of inertia w.r.t. the principal axes and $\lambda_1, \lambda_2, \lambda_3$ be the coordinates of the center of mass in the same axes. The Euler equations of motion have the form
\begin{equation}
\label{eq55}
   \begin{aligned}
        &A\dot p + (C - B)qr = \mu(\lambda_3\gamma_2 - \lambda_2\gamma_3),\\
        &B\dot q + (A - C)rp = \mu(\lambda_1\gamma_3 - \lambda_3\gamma_1),\\
        &C\dot r + (B - A)pq = \mu(\lambda_2\gamma_1 - \lambda_1\gamma_2),\\
        &\dot\gamma_2 + r\gamma_1 - p\gamma_3 = 0,\\
        &\dot\gamma_1 + q\gamma_3 - r\gamma_2 = 0,\\
        &\dot\gamma_3 + p\gamma_2 - q\gamma_1 = 0.
    \end{aligned}
\end{equation}
Here $p,q,r$ are the components of the angular velocity in the principal axes, $\gamma_1, \gamma_2, \gamma_3$ are the coordinates of the vertical unit vector in the same axes. This system has the following first integrals:
\begin{equation}
\begin{aligned}
    &\frac{1}{2}(Ap^2 + Bq^2 + Cr^2) + \mu(\lambda_1\gamma_1 + \lambda_2\gamma_2 + \lambda_3\gamma_3) = h,\\
    &Ap\gamma_1 + Bq\gamma_2 + Cr\gamma_3 = c, \quad \gamma_1^2 + \gamma_2^2 + \gamma_3^2 = 1.
\end{aligned}
\end{equation}

Let $c = 0$, by $M_h$ we denote the three-dimensional non-critical connected component of the level set of the first integrals (with energy $h$). $M_h$ is a smooth manifold. System  (\ref{eq55}) has an invariant measure, which immediately follows from the Liouville  theorem. The density of this measure is constant and, without loss of generality, we can assume that it equals $1$. Then there also exists an invariant measure on the level set of the first integrals (see, e.g., \cite{bolsinov2011hamiltonization}):

\begin{theorem}
Let us have a system $\dot x = v(x)$ on an $n$-dimensional manifold $M$. Suppose that the system has an invariant measure $\mu$ with a smooth density and has $k$ first integrals  $F_1, ..., F_k$. Let $N$ be a non-critical level set of the first integrals. Then the restriction of the initial system on $N$ also has an invariant measure and this measure is defined by an $(n-k)$-form $\nu$
$$
\nu\wedge dF_1 \wedge ... \wedge dF_k = \mu.
$$
\end{theorem}
One can show that the following lemma holds
\begin{lemma}
Function $ f = (p\gamma_1 + q\gamma_2)/(\gamma_1^2 + \gamma_2^2)$ is Lebesgue integrable on $M_h$.
\end{lemma}
\begin{proof}
It is sufficient to consider $f$ only in the vicinities of the points where $\gamma_1 = \gamma_2 = 0$ and $\gamma_3 = 1$, since $f$ is continuous everywhere else.
\end{proof}
\begin{lemma}
The space average of $f$ on $M_h$ is zero, i.e.
$$
\int\limits_{M_h} f d\nu = 0.
$$
Here we integrate w.r.t. the invariant measure $\nu$ on $M_h$.
\end{lemma}
\begin{proof}
Suppose that a point $(p,q,r,\gamma_1,\gamma_2,\gamma_3)$ is in $M_h$. Let us show that the point $(-p,-q,-r,\gamma_1,\gamma_2,\gamma_3)$ is also in $M_h$. Since $c = 0$, then the both points are in the level set corresponding to the energy $h$. It is sufficient to prove that they belong to the same connected component. We can consider $\gamma_1, \gamma_2, \gamma_3$ as given parameters. From the energy integral, we obtain that $p,q,r$ belong to some ellipsoid. From the area integral we have that $p,q,r$ lie in a plane which passes through the origin. Finally, the points $(p,q,r,\gamma_1,\gamma_2,\gamma_3)$ and $(-p,-q,-r,\gamma_1,\gamma_2,\gamma_3)$ can be joined by a continuous path that lies in $M_h$.

Since $f(p,q,r,\gamma_1,\gamma_2,\gamma_3) = -f(-p,-q,-r,\gamma_1,\gamma_2,\gamma_3)$ and the density of measure $\nu$ is the same for points $(p,q,r,\gamma_1,\gamma_2,\gamma_3)$ and $(-p,-q,-r,\gamma_1,\gamma_2,\gamma_3)$ (it follows from the symmetry of $M_h$ and the symmetry of the vector field of our system), then the considered integral equals zero.
\end{proof}
\begin{proposition}
The average value of the main motions of the precession angle on $M_h$ is zero.
\end{proposition}
\begin{proof}
The change of the precession angle is defined by the equation $\dot\psi = f(p,q,r,\gamma_1,\gamma_2,\gamma_3)$. From the Birkhoff-Khinchin theorem, we have that for almost all $(p,q,r,\gamma_1,\gamma_2,\gamma_3)$, the time average $f^*(p,q,r,\gamma_1,\gamma_2,\gamma_3)$ is correctly defined. The time average equals to the main motion of the precession angle $\Lambda(p,q,r,\gamma_1,\gamma_2,\gamma_3) = \lim\limits_{t\to+\infty}\psi(t)/t$. Moreover, from the same theorem, we have
$$
\int\limits_{M_h} \Lambda d\nu = 0.
$$
\end{proof}

\section{Kovalevskaya top}
\subsection{The case of zero area integral}

\noindent By the change of coordinates $p \mapsto p\sqrt{\mu}$, $q \mapsto q\sqrt{\mu}$, $r \mapsto r\sqrt{\mu}$, $t \mapsto t/\sqrt{\mu}$ system (\ref{eq1}) can be simplified to the form
\begin{equation}
\label{eq3}
    \begin{aligned}
        &2\dot p - qr = 0, \quad 2\dot q + rp = \gamma_3, \quad \dot r = -\gamma_2,\\
        &\dot\gamma_1 = r\gamma_2 - q\gamma_3,\quad \dot\gamma_2 = p\gamma_3 - r\gamma_1, \quad \dot\gamma_3 = q\gamma_1 - p\gamma_2.
    \end{aligned}
\end{equation}
The first integrals take the form
\begin{equation}
\label{kov_ints}
    \begin{aligned}
        &p^2 + q^2 + \frac{1}{2}r^2 + \gamma_1 = h, \quad 2(p\gamma_1 + q\gamma_2) + r\gamma_3 = c,\\
        &\gamma_1^2 + \gamma_2^2 + \gamma_3^2 = 1, \quad (p^2 - q^2 - \gamma_1)^2 + (2pq - \gamma_2)^2 = k^2.
    \end{aligned}
\end{equation}
The bifurcation diagram of (\ref{eq3}) for $c = 0$ is presented in Fig.~1 \cite{kharlamov1988,borisov2001rigid,gashenenko2012classical}. The critical values of  $h$ and $k^2$, for which the functions of the first integrals (\ref{kov_ints}) become dependent, correspond to the Appelrot classes. In these cases, the integration is simplified and can be carried out in detail \cite{golub53,appelroth1910cas,appelroth1911cas2}. Therefore, the dynamics on the bifurcation set is relatively well studied.
\begin{figure}[h!]
\centering
\includegraphics[width=0.99\linewidth]{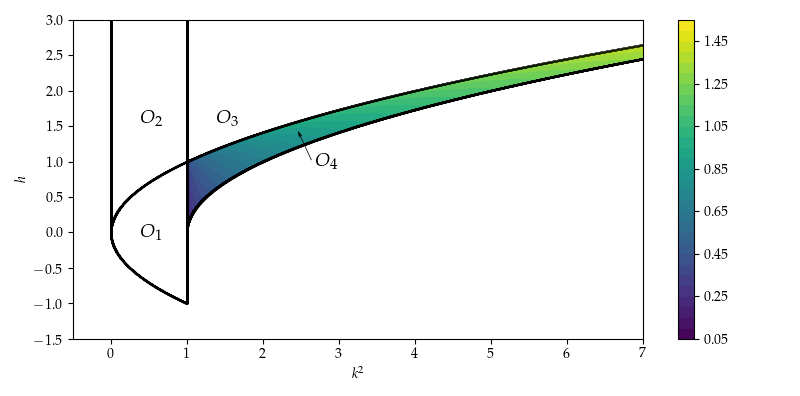}
\caption{Bifurcation diagram of the Kovalevskaya top for the case of zero area integral. Possible values of $h$ and $k$ lie in the closure of $O_1$, $O_2$, $O_3$ and $O_4$. Region where $\Lambda \ne 0$ is highlighted.}
\label{fig:fig}
\end{figure}
\begin{remark}
For the Kovalevskaya top, system (\ref{eq3}) can be presented as the system of two equations
$$
\dot s_1 = \frac{\sqrt{\Phi_{h,k,c}(s_1)}}{s_1 - s_2}, \quad \dot s_2 = \frac{\sqrt{\Phi_{h,k,c}(s_2)}}{s_2 - s_1},
$$
where $\Phi_{h,k,c}$ is a fifth-degree polynomial that depends on $h$, $k$ and $c$. The bifurcation values of $h$, $k$, $c$ correspond to the cases when the polynomial have multiple roots. In particular, this allows us to obtain the solution by means of elliptic functions. A more detailed exposition can be found in \cite{golub53,borisov2001rigid}.
\end{remark}

Below we consider a more typical situation when, for given $h$ and $k$, the first integrals are independent. We show that, when the constant of the area integral equals zero ($c = 0$), the mean motion is also zero provided $h$, $k$ is in $O_1 \cup O_2 \cup O_3$ and the invariant torus is non-resonant (Fig. 1)

First, we present some auxiliary results that will be used below. More details can be found in \cite{kozlov2000,kharlamov1988,arnold1967probiernes}.
\begin{lemma}
Suppose that for given initial data, $c = 0$ and $h$, $k$ belong to one of the open sets $O_1, O_2, O_3, O_4$. Then the solution of (\ref{eq3}) does not passes through the points where $\gamma_3 = \pm 1$.
\end{lemma}
\begin{lemma}
Let us have an invariant torus that belong to one of the sets $O_1, O_2, O_3, O_4$ and $\varphi_1$, $\varphi_2$ are angular coordinates on this torus. Then the function $\Psi(\varphi_1,\varphi_2) = \dot\psi$, that defines the change of the precession angle, is smooth.
\end{lemma}
\begin{lemma}
In $O_1$ the level set of the first integrals is a two-dimensional torus. In $O_2$, $O_3$ and $O_4$ the level set is two two-dimensional tori.
\end{lemma}
\begin{lemma}
Let $\mathbb{T}^2$ be a two-dimensional invariant torus of the Kovalevskaya top and $\varphi_1, \varphi_2$ are angle variables on it. Let the restriction of a function $f(p,q,r,\gamma_1,\gamma_2,\gamma_3)$ on $\mathbb{T}^2$ is Lebesgue integrable. Then
$$
\int\limits_{\mathbb{T}^2} f(\varphi_1, \varphi_2) d\varphi_1 d\varphi_2 = \int\limits_{\mathbb{T}^2} \frac{f}{V} d\sigma,
$$
where $\sigma$ is the surface element of the manifold embedded in $\mathbb{R}^6$,  $V$ is the volume of the four-dimensional span of vectors $\mathrm{grad}\,I_i$, $i = 1,2,3,4$. Here $I_i$ are the functions of the first integrals (left-hand sides of (\ref{kov_ints})).
\end{lemma}
\begin{lemma}
Let $\alpha \colon p, q, r, \gamma_1, \gamma_2, \gamma_3 \mapsto -p, -q, r, \gamma_1, \gamma_2, -\gamma_3$. Then for the Kovalevskaya top $V(p,q,r,\gamma_1,\gamma_2,\gamma_3) = V(\alpha(p,q,r,\gamma_1,\gamma_2,\gamma_3))$, i.e. the volume of the span is preserved under the map $\alpha$.
\end{lemma}
\begin{proposition}
Suppose that, for given initial data, $h$ and $k^2$ belong to $O_1$ and the invariant torus is non-resonant. Then the mean motion of the line of nodes is zero, i.e. $\Lambda = 0$.
\end{proposition}
\begin{proof}
The change of the precession angle is described as follows
\begin{equation}
\label{eq9}
    \dot \psi = \frac{p\gamma_1 + q\gamma_2}{\gamma_1^2 + \gamma_2^2} = \frac{1}{2}\frac{r\gamma_3}{\gamma_3^2 - 1}.
\end{equation}
Consider the projection of the invariant torus onto the plane with coordinates $r$ and $\gamma_3$. If some point $(p,q,r,\gamma_1,\gamma_2,\gamma_3)$ lie on the invariant torus, then the points $(p,q,-r,\gamma_1,\gamma_2,-\gamma_3)$, $(-p,-q,-r,\gamma_1,\gamma_2,\gamma_3)$, $(-p,-q,r,\gamma_1,\gamma_2,-\gamma_3)$ also belong to the same torus. Therefore, if some subset of the torus is projected onto the quadrant where $r > 0$ and $\gamma_3 > 0$, then the same subsets (up to symmetries) are projected onto other quadrants. Since the level set of the first integrals is one torus, then, from Lemmas 6 and 7, we obtain that the space average of $\Psi(\varphi_1, \varphi_2)$ is zero. Since the torus is non-resonant, then the flow is ergodic and $\Lambda = 0$.
\end{proof}
\begin{proposition}
Suppose that, for given initial data, $h$ and $k^2$ belong to $O_2$ or $O_3$ and the invariant torus is non-resonant. Then the mean motion of the line of nodes is zero, i.e. $\Lambda = 0$.
\end{proposition}
\begin{proof}
For any pair $h$, $k^2$ in $O_2$ and $O_3$, the level set of the first integrals is two tori. Therefore, we cannot directly apply the above arguments: the level set may have the required symmetries, yet each torus may be non-symmetrical.

We now show that the projection onto the plane $(r, \gamma_3)$ of the invariant torus does not intersect the line $r = 0$. For this we will show that for $r = 0$ we always have $h^2 \leqslant k^2$, i.e.
$$
(p^2 + q^2 + \gamma_1)^2 \leqslant (p^2 - q^2 - \gamma_1)^2 + (2pq - \gamma_2)^2.
$$
This inequality is equivalent to the following
$$
2p^2(2q^2 + 2\gamma_1) \leqslant (2pq - \gamma_2)^2.
$$
We obtain
$$
4\gamma_1 p^2 + 4pq\gamma_2 - \gamma_2^2 - 2q^2p^2 \leqslant 0.
$$
Since the area integral is zero, we can put
 $p\gamma_1 = -q\gamma_2$. Finally,
$$
- \gamma_2^2 - 2q^2p^2 \leqslant 0.
$$
Therefore, the initial inequality also holds. We have proved that the projections of two invariant tori are symmetric w.r.t. the line $r = 0$ and do not intersect this line. Moreover, they are symmetric w.r.t. the line $\gamma_3 = 0$ and always have non-empty intersection with it (otherwise, there will be at least four invariant tori). Now we can apply the arguments of the proof of Proposition 2 and obtain $\Lambda = 0$.
\end{proof}
\begin{remark}
Note that results similar to Propositions 2 and 3 have been obtained in \cite{gashenenko2012classical} applying the so-called hodograph method. However, in \cite{gashenenko2012classical}, only the main motion is studied.
\end{remark}

\begin{remark}
If non-resonant tori are dense in $O_1$, $O_2$ and $O_3$ then, since $\Lambda$ is a continuous function, we can conclude that $\Lambda = 0$ for all tori, not only for non-resonant.
\end{remark}

In conclusion, we present some numerical results concerning the behaviour of $\Lambda$ in $O_4$. For a given pair $(k^2,h) \in O_4$, we choose some initial conditions from the corresponding invariant torus. Along the solution we never approach the positions where $\gamma_3 = \pm 1$. Therefore, the precession angle is defined correctly. In order to obtain $\Lambda$, we numerically integrate system (\ref{eq3}) together with the equation for $\dot\psi$ and apply the least squares method.

From the numerical results, it can be seen that $\Lambda \ne 0$ everywhere in $O_4$. Moreover, $\Lambda$ is increasing in $O_4$ as $k^2$ increases. Since $\Lambda = 0$ in $O_1$, $O_2$ and $O_3$, we obtain that function $\Lambda$ has discontinuities at the border $\partial O_4$. We illustrate this discontinuity in Fig. 2 by showing the change in the topology of the region covered by the trajectory of solution. When this region becomes a ring around the vertical axis, $\Lambda$ becomes non-zero. At the same time, based on the calculations, it is possible to conclude that $\Lambda$ changes continuously at the border between regions $O_1$ and $O_4$.
\begin{figure*}[t]
\centering
\begin{subfigure}{.5\textwidth}
  \centering
  \includegraphics[width=0.99\linewidth]{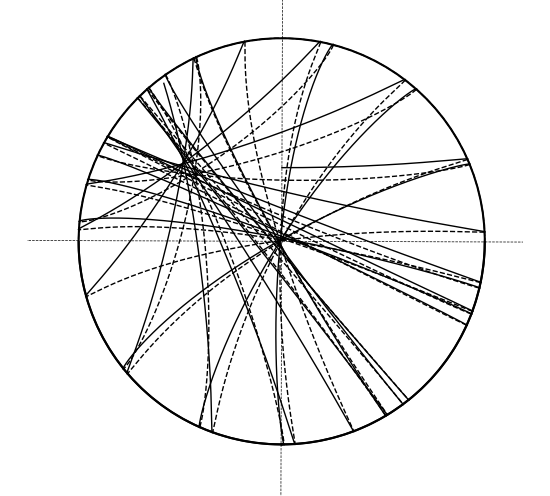}
  \caption{$\Lambda = 0$}
  \label{fig:sfig1}
\end{subfigure}%
\begin{subfigure}{.5\textwidth}
  \centering
  \includegraphics[width=0.99\linewidth]{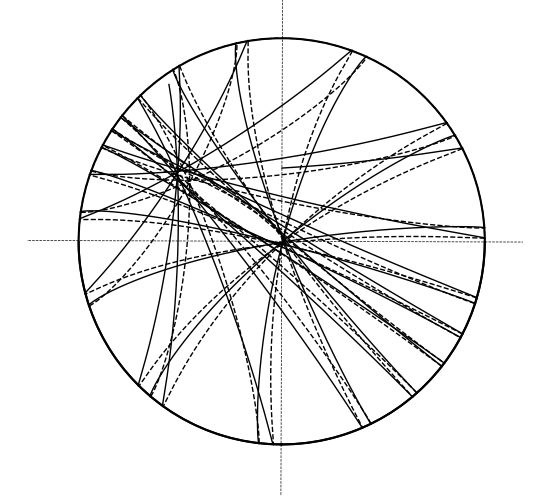}
  \caption{$\Lambda \ne 0$}
  \label{fig:sfig1}
\end{subfigure}
\caption{Trajectories of the end of the axis of dynamical symmetry for $O_3$ (a) and $O_4$ (b).}
\label{fig:fig}
\end{figure*}
\begin{figure*}[h]
\centering
\begin{subfigure}{.49\textwidth}
  \centering
  \includegraphics[width=0.99\linewidth]{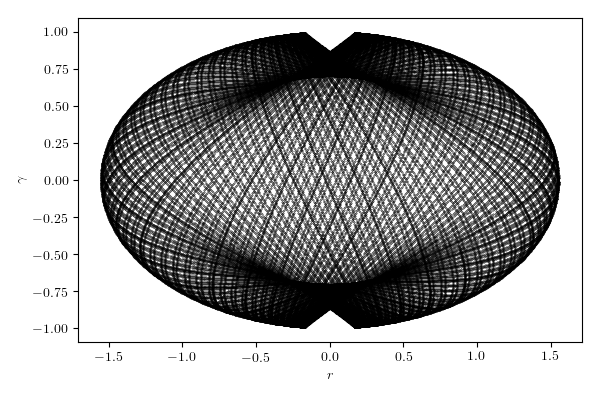}
  \caption{$O_1$: $h = 1/2$, $k^2 = 1/2$}
  \label{fig:sfig1}
\end{subfigure}%
\begin{subfigure}{.49\textwidth}
  \centering
  \includegraphics[width=0.99\linewidth]{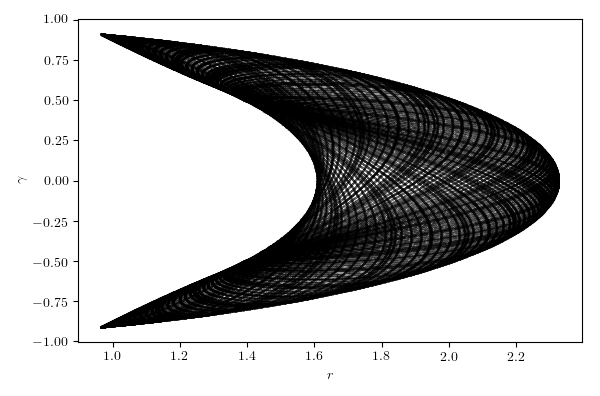}
  \caption{$O_2$: $h = 1$, $k^2 = 3/2$}
  \label{fig:sfig1}
\end{subfigure}\\
\begin{subfigure}{.49\textwidth}
  \centering
  \includegraphics[width=0.99\linewidth]{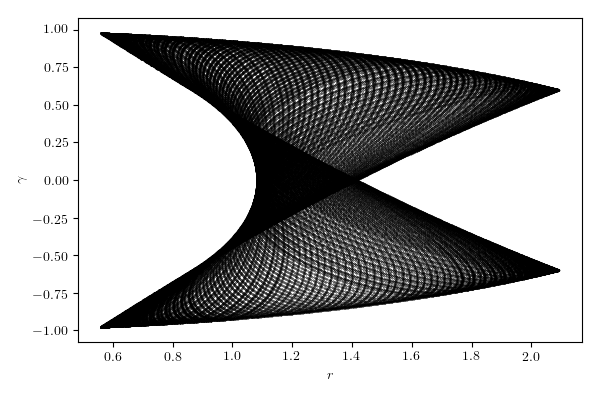}
  \caption{$O_3$: $h = 1/2$, $k^2 = 1/2$ }
  \label{fig:sfig1}
\end{subfigure}%
\begin{subfigure}{.49\textwidth}
  \centering
  \includegraphics[width=0.99\linewidth]{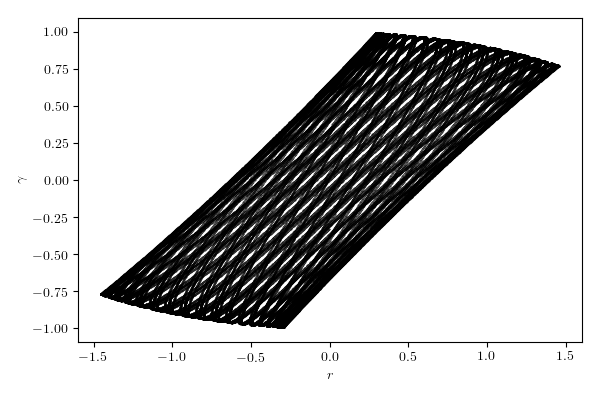}
  \caption{$O_4$: $h = 1$, $k^2 = 3/2$}
  \label{fig:sfig1}
\end{subfigure}
\caption{Examples of the projections of invariant tori on the plane $(r, \gamma_3)$. For   $O_1$, $O_2$ and $O_3$ the mean motions are zero.}
\label{fig:fig}
\end{figure*}
We note that the rigorous proof of the fact that $\Lambda \ne 0$ is less trivial than the proofs of the above propositions. Indeed, above we substantially use the symmetries of the invariant torus. When we consider an invariant torus from $O_4$, based on the numerical results, we can conclude that the projection have only one symmetry  $r \mapsto -r$, $\gamma_3 \mapsto -\gamma_3$ (Fig.~3). Therefore, we can expect that the value of the corresponding integral is not zero (which is confirmed by the numerical analysis), yet the proof requires technically complicated calculations.

In \cite{kharlamov1988,gashenenko2012classical}, it was shown that each point in $O_4$ corresponds to a couple of Liouville tori. Since equations (\ref{kov_ints}) (for $c = 0$) are symmetric under the mapping $p \mapsto -p$, $q \mapsto -q$, $r \mapsto -r$, i.e. the projection of the level set onto the plane $(r, \gamma_3)$ is symmetric w.r.t. the line $\gamma_3$, we obtain that for any mean motion $\Lambda$ we also have a solution with its mean motion equals $-\Lambda$. Moreover, $k^2$, $h$ are the same for these two solutions.

\subsection{The case of non-zero area integral}
\noindent Let us now consider the case when $c \ne 0$. Taking into account the symmetries of the system, we will consider only the case $c > 0$. The full classification of the Liouville foliation of the Kovalevskaya top was presented by M.\,P.\,Kharlamov in \cite{kharlamov1988}. It was shown that there are five different types of two-dimensional bifurcation diagrams depending on the value of the area integral. Consequently, for system (\ref{eq3}), there are four critical values of $c$ for which the type of the diagram changes: $c^*_0 = 0$, $c^*_1 = \sqrt{2}$, $c^*_2 = 4/3^{3/4} \approx 1.75$, $c^*_3 = 2$.

We will denote the regions of the three-dimensional diagram by $O_1$, $O_2$, $O_3$, $O_4$ и $O_5$. The same notations will be used for the corresponding two-dimensional sections of these regions (as it was used for the case $c=0$). Lemma 5 still holds for the three-dimensional bifurcation regions. Moreover, for any point in $O_5$ we have four invariant tori \cite{kharlamov1988,gashenenko2012classical}.

Now consider the projections of the invariant tori onto the plane $(p,q)$ for various regions of the bifurcation diagram. It was shown \cite{kharlamov1988,appelroth1910cas,appelroth1911cas2} that for $O_1$ this projection is always a curvilinear quadrangle, the projection of the only invariant torus. For $O_2$ and $O_3$ the projection of the two invariant tori is always one ring. The invariant tori from $O_4$ are projected into two curvilinear quadrangles. For $O_5$, the four invariant tori are projected into two rings. In other words, each ring on the plane $(p,q)$ corresponds to a pair of invariant tori, each curvilinear quadrangle corresponds to a single invariant torus.

Therefore, for any point in $O_1$, we have a single value of $\Lambda$. For $O_2$, $O_3$, $O_4$, for given $c$, $h$, $k^2$, we could have two different $\Lambda$. For $O_5$ it could be four different values of $\Lambda$. However, based on the numerical results, we can conclude that it is not the case: for two tori that are projected into the same ring we always have the same $\Lambda$. It means that, for a given point in $O_1$, $O_2$ or $O_3$, we always have a single value of $\Lambda$, for $O_4$ and $O_5$ we can have two different values of $\Lambda$ at each point.

Everywhere below we present results for the following values of the area integral: $c_1 \approx 1.03$, $c_2 \approx 1.71$, $c_3 \approx 1.88$, $c_4 = 3$. Here, $c_i \in [c^*_{i-1}, c^*_i]$ for $1 \leqslant i \leqslant 3$ and $c_4 > c^*_3$. We choose these values for the sake of easy visualization. Since we are mostly interested in studying the qualitative distribution of $\Lambda$, below we never mention the specific values of $c_1$, $c_2$, $c_3$ и $c_4$.

\begin{figure*}[t]
\centering
\begin{subfigure}{.5\textwidth}
  \centering
  \includegraphics[width=0.99\linewidth]{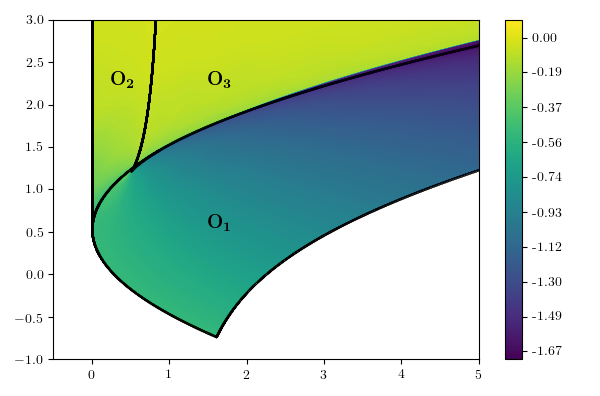}
\caption{$c = c_1$}  \label{fig:sfig1}
\end{subfigure}%
\begin{subfigure}{.5\textwidth}
  \centering
  \includegraphics[width=0.99\linewidth]{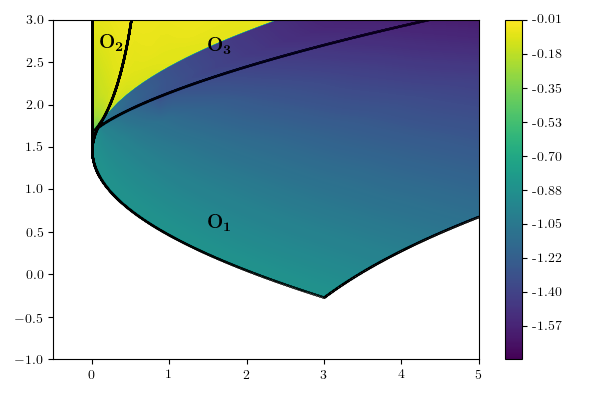}
\caption{$c = c_2$}
  \label{fig:sfig1}
\end{subfigure}\\
\begin{subfigure}{.5\textwidth}
  \centering
  \includegraphics[width=0.99\linewidth]{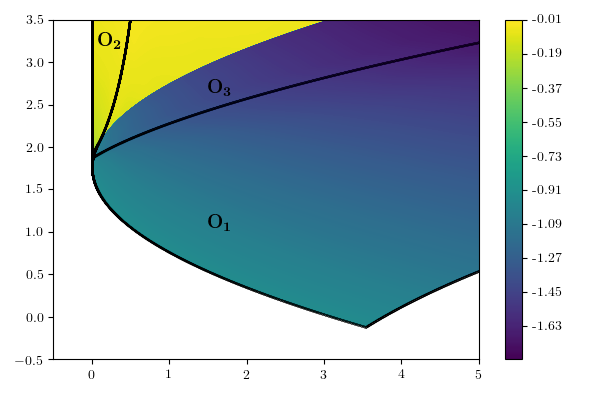}
\caption{$c = c_3$}
\label{fig:sfig1}
\end{subfigure}%
\begin{subfigure}{.5\textwidth}
  \centering
  \includegraphics[width=0.99\linewidth]{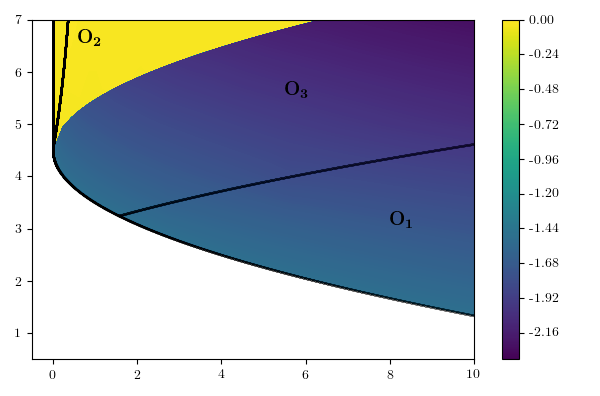}
\caption{$c = c_4$}
\label{fig:sfig1}
\end{subfigure}
\caption{Bifurcation diagrams of the Kovalevskaya top. All qualitatively different types are presented (except for $c = 0$). Solid lines are for the bifurcation values of first integrals.}
\label{fig:fig}
\end{figure*}

The results are presented in Fig.~4. Similarly to the case $c=0$, the value of $\Lambda$ changes discontinuously at some points of the diagrams. In particular, $\Lambda$ is discontinuous at some points satisfying  $h^2 = c^2/2 + k$. In contrast to the case of zero area integral, these points are not always the points of bifurcation or our system, i.e. the topology of the Liouville foliation does not change as we pass through these values. However, these points have the following important property.

\begin{proposition}
Suppose that for a solution of (\ref{eq3}) it holds that $h^2 \ne c^2/2 + k$. Then the trajectory of this solution is separated from the points where $\gamma_3 = \pm 1$. In particular, the angle of precession $\psi$ is correctly defined for it.
\end{proposition}

In other words, from the proposition we have that for all solutions that possibly can pass through the positions $\gamma_3 = \pm 1$, we always have $h^2 = c^2/2 + k$. Let us consider a one-parameter family of solutions. Suppose that $c$ is fixed and the corresponding curve in the plane $(k^2, h)$ intersects the curve $h^2 = c^2/2 + k$. Then, based on the numerical results, we can conclude that one of the possible scenarios in which $\Lambda$ becomes discontinuous is when we have a solution that passes through the points $\gamma_3 = \pm 1$ in our one-parameter family. However, it is not the only possibility and $\Lambda$ may become discontinuous at points for which $h^2 \ne c^2/2 + k$. Moreover, $\Lambda$ is continuous at some points of the curve $h^2 = c^2/2 + k$.

For $c=c_1$, a more detailed exposition of a small region of the bifurcation in Fig.~5. It can be seen that for each point in $O_4$ we have two different values of $\Lambda$. For one family of invariant tori, $\Lambda$ is continuous at the curve separating $O_1$ and $O_4$.

\begin{figure*}[t]
\centering
\begin{subfigure}{.5\textwidth}
  \centering
  \includegraphics[width=0.99\linewidth]{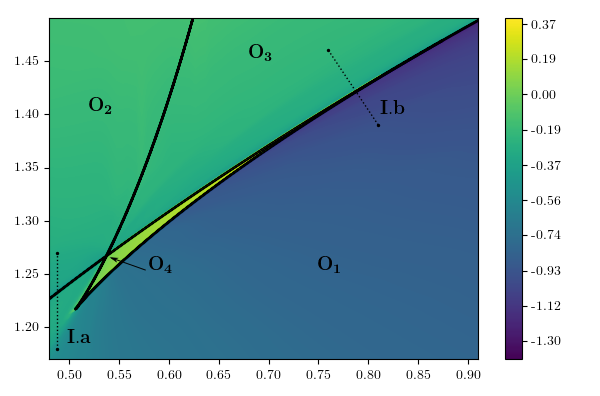}
\caption{$c = c_1$}  \label{fig:sfig1}
\end{subfigure}%
\begin{subfigure}{.5\textwidth}
  \centering
  \includegraphics[width=0.99\linewidth]{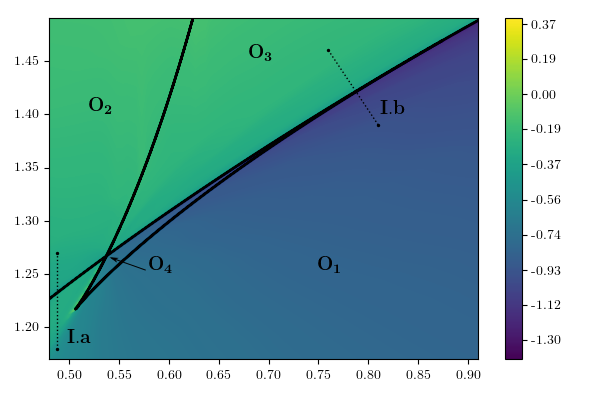}
\caption{$c = c_1$}
  \label{fig:sfig1}
\end{subfigure}
\caption{Enlarged regions of the bifurcation diagram of the Kovalevskaya top for $c = c_1$.}
\label{fig:fig}
\end{figure*}
\begin{figure*}[h]
\centering
\begin{subfigure}{.5\textwidth}
  \centering
  \includegraphics[width=0.99\linewidth]{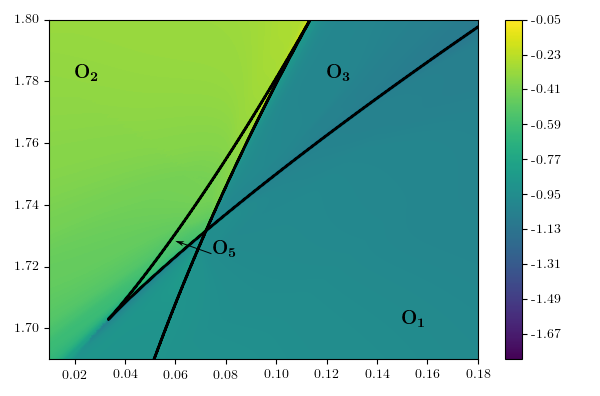}
\caption{$c = c_2$}  \label{fig:sfig1}
\end{subfigure}%
\begin{subfigure}{.5\textwidth}
  \centering
  \includegraphics[width=0.99\linewidth]{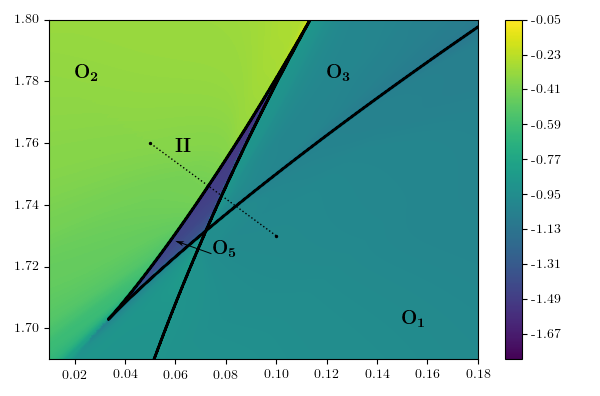}
\caption{$c = c_2$}
  \label{fig:sfig1}
\end{subfigure}
\caption{Enlarged regions of the bifurcation diagram of the Kovalevskaya top for $c = c_2$.}

\label{fig:fig}
\end{figure*}
\begin{figure*}[h]
\centering
\begin{subfigure}{.5\textwidth}
  \centering
  \includegraphics[width=0.99\linewidth]{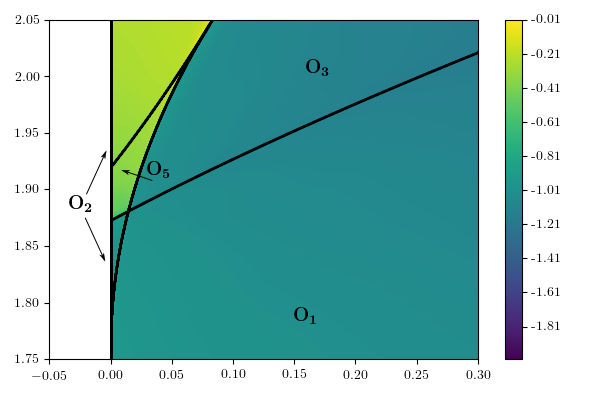}
\caption{$c = c_3$}  \label{fig:sfig1}
\end{subfigure}%
\begin{subfigure}{.5\textwidth}
  \centering
  \includegraphics[width=0.99\linewidth]{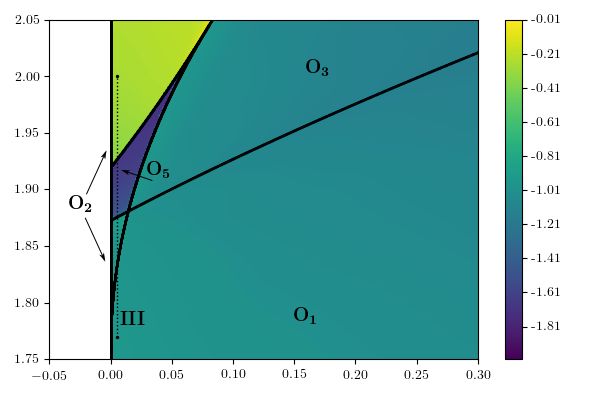}
\caption{$c = c_3$}
  \label{fig:sfig1}
\end{subfigure}
\caption{Enlarged regions of the bifurcation diagram of the Kovalevskaya top for $c = c_3$.}
\label{fig:fig}
\end{figure*}
\begin{figure*}[t]
\centering
\begin{subfigure}{.5\textwidth}
  \centering
  \includegraphics[width=0.99\linewidth]{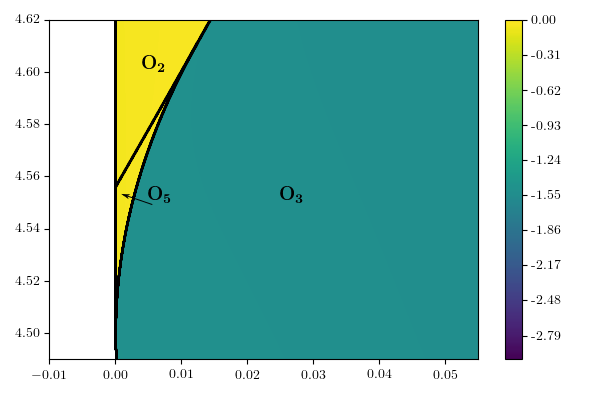}
\caption{$c = c_4$}  \label{fig:sfig1}
\end{subfigure}%
\begin{subfigure}{.5\textwidth}
  \centering
  \includegraphics[width=0.99\linewidth]{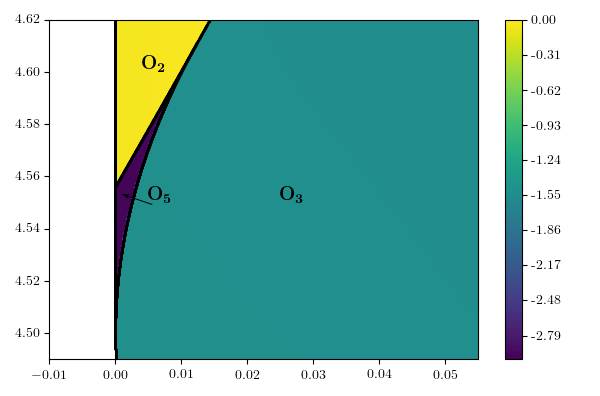}
\caption{$c = c_4$}
  \label{fig:sfig1}
\end{subfigure}
\caption{Enlarged regions of the bifurcation diagram of the Kovalevskaya top for $c = c_4$.}
\label{fig:fig}
\end{figure*}
\begin{figure*}[h]
\centering
\begin{subfigure}{.5\textwidth}
  \centering
  \includegraphics[width=0.99\linewidth]{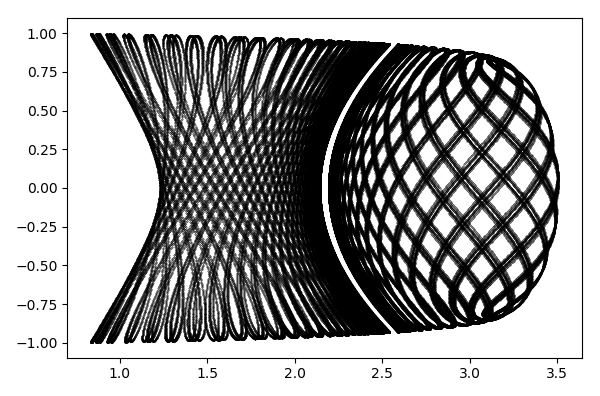}
  \caption{$\displaystyle{h > \frac{3}{2}|2k|^{2/3} + 1}$}
  \label{fig:sfig1}
\end{subfigure}%
\begin{subfigure}{.5\textwidth}
  \centering
  \includegraphics[width=0.99\linewidth]{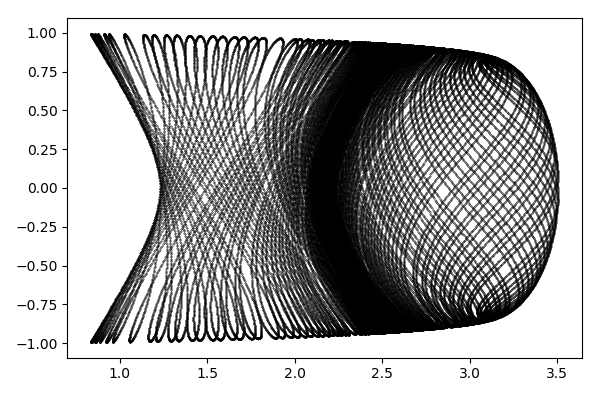}
  \caption{$\displaystyle{h < \frac{3}{2}|2k|^{2/3} + 1}$}
  \label{fig:sfig1}
\end{subfigure}%
\caption{Projections of solutions on the plane $(r, \gamma_3)$ for the Goryachev-Chaplygin top. The projections are symmetric w.r.t. the line $\gamma_3 = 0$. It is seen that one invariant torus splits into two tori as we cross the bifurcation curve.}
\label{fig:fig}
\end{figure*}

The results for the other cases are presented in Figs. 6, 7 and 8. Again, we see that, depending on the choice of the family of invariant tori in $O_5$, $\Lambda$ can be continuous or discontinuous at the points between $O_2$ and $O_5$.

For more detailed exposition of the results, we also present several plots that show how $\Lambda$ along sections in the bifurcation diagrams (Fig.~10).

For each point of these plots, we choose the total time of numerical integration $T$ in such a way that the following inequality holds
$$
\left|\frac{\Lambda(T) - \Lambda(2T)}{\Lambda(T)}\right| < 0.0005,
$$
i.e. when we consider a time interval that is two times longer than $[0, T]$, the value of $\Lambda$ remains almost unchanged.

It is seen that along section I.a ($c=c_1$), $\Lambda$ changes continuously and for section I.b the plot has a discontinuity. Similarly, the plots have discontinuities for sections II ($c=c_2$) and III ($c=c_3$).

The plots are scaled for the sake of visualization, i.e. the starting and ending points of the plots are indeed different. Segment I.a connects points $k^2 = 0.488$, $h = 1.18$ and $k^2 = 0.488$, $h = 1.27$, segment I.b connects points for which $k^2 = 0.81$, $h = 1.39$ and $k^2=0.76$, $h = 1.46$. Segment II connects $k^2=0.1$, $h = 1.73$ and $k^2 = 0.05$, $h = 1.76$. For the starting and ending points of segment III we have $k^2 = 0.005$, $h = 1.77$ and $h = 2.0$.

\begin{figure}[h]
\centering
\includegraphics[width=0.99\linewidth]{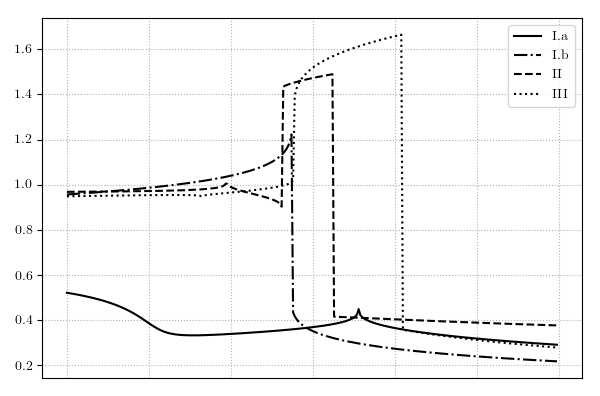}
\caption{Values of $|\Lambda|$ along the sections on the bifurcation diagrams.}
\label{fig:fig}
\end{figure}
\begin{figure}[h]
  \centering
  \includegraphics[width=0.99\linewidth]{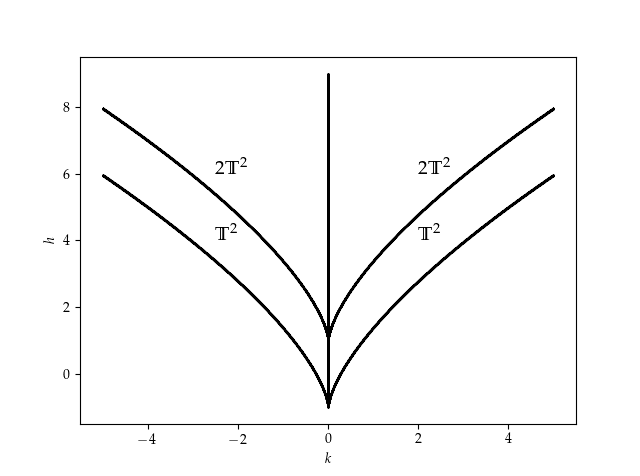}
  \caption{Bifurcation diagram for the Goryachev-Chaplygin top.}
  \label{fig:sfig1}
\end{figure}

\section{Goryachev-Chaplygin top}

\noindent This integrable case can be presented (e.g., \cite{chaplygin1901new,zbMATH02669215,goryachev1915new}) by the following system
\begin{equation}
\label{eq10}
    \begin{aligned}
        &4\dot p - 3qr = 0, \quad 4\dot q + 3rp = \gamma_3, \quad \dot r = -\gamma_2,\\
        &\dot\gamma_1 = r\gamma_2 - q\gamma_3,\quad \dot\gamma_2 = p\gamma_3 - r\gamma_1, \quad \dot\gamma_3 = q\gamma_1 - p\gamma_2.
    \end{aligned}
\end{equation}
The first integrals have the form
\begin{equation}
\label{gor_ints}
    \begin{aligned}
        &2(p^2 + q^2) + \frac{1}{2}r^2 + \gamma_1 = h, \quad 4(p\gamma_1 + q\gamma_2) + r\gamma_3 = 0\\
        &r(p^2 + q^2) - p\gamma_3 = k, \quad \gamma_1^2 + \gamma_2^2 + \gamma_3^2 = 1.
    \end{aligned}
\end{equation}

In this section we again study the motion of the line of nodes, yet, in this aspect, the Goryachev-Chaplygin case is completely different comparing to the Kovalevskaya top. First, the measure of tori that intersect the subspace $\gamma_3 = \pm 1$ is non-zero and we need to define what we understand by $\psi(t)$ for the moments of time when the solution passes through the points where $\gamma_3 = \pm 1$.

Suppose that $\gamma_3 = \pm 1$ holds for $t = t'$. Then the value $\dot \psi$ is not defined. However, from the l'Hopital's rule we can show that the following limit is correctly defined $\lim\limits_{t \to t'} \Psi(t)$. Here $\Psi(t)$ is the right-hand side of (\ref{eq9}) considered along a solution. Hence, we can put
$$
\psi(t) = \psi(t_0) + \int\limits_{t_0}^t \Psi(t)\, dt,
$$
where $\Psi(t)$ is defined by the l'Hopital's rule.

\begin{lemma}\cite{kozlov2000}
Suppose that $h < 2k^2$. Then the solution is separated from the positions where $\gamma_3 = \pm 1$.
\end{lemma}
Let $\varphi_1$ and $\varphi_2$ be angular variables on the two-dimensional invariant torus of system (\ref{eq10}) for fixed values of first integrals (\ref{gor_ints}). Let $\Psi(\varphi_1, \varphi_2)$ be the function that defines the change of the precession angle along a trajectory belonging to the invariant torus, i.e. similarly to the Kovalevskaya top, we have
$$
\dot \psi = \Psi(\varphi_1, \varphi_2).
$$
\begin{lemma}\cite{kozlov2000}
Suppose that $h \ne 2k^2$. Then function $\Psi$ is Lebesgue integrable.
\end{lemma}
\begin{lemma}\cite{kozlov2000}
Suppose $h \ne 2k^2$ and the invariant torus is non-resonant. Then the line of nodes has a main motion which does not depend on the initial data on the torus. If the torus is resonant, the line of nodes has a mean motion.
\end{lemma}
\begin{remark}
Note that from Lemma 10 it follows that a main motion exists for any initial condition on the torus. From the ergodic theorem we obtain the existence of a main motion for almost all initial data.
\end{remark}
For the Goryachev-Chaplygin case the following results, similar to Lemmas 6 and 7, also hold
\begin{lemma}
Let $\mathbb{T}^2$ be a two-dimensional invariant torus of the Goryachev-Chaplygin top and $\varphi_1, \varphi_2$ are angle variables on it. Let the restriction of a function\\ $f(p,q,r,\gamma_1,\gamma_2,\gamma_3)$ on $\mathbb{T}^2$ is Lebesgue integrable. Then
$$
\int\limits_{\mathbb{T}^2} f(\varphi_1, \varphi_2) d\varphi_1 d\varphi_2 = \int\limits_{\mathbb{T}^2} \frac{f}{V} d\sigma,
$$
where $\sigma$ is the surface element of the manifold embedded $\mathbb{R}^6$,  $V$ is the volume of the four-dimensional span of vectors $\mathrm{grad}\,I_i$, $i = 1,2,3,4$. Here $I_i$ are the functions of the first integrals (left-hand sides of (\ref{gor_ints})).
\end{lemma}
\begin{lemma}
Let $$\alpha \colon p, q, r, \gamma_1, \gamma_2, \gamma_3 \mapsto -p, -q, r, \gamma_1, \gamma_2, -\gamma_3.$$ Then for the Goryachev-Chaplygin top $$V(p,q,r,\gamma_1,\gamma_2,\gamma_3) = V(\alpha(p,q,r,\gamma_1,\gamma_2,\gamma_3)),$$ i.e. the volume of the span is preserved under the map $\alpha$.
\end{lemma}

Since the Goryachev-Chaplygin top is integrable only when the area integral is zero, below we consider the bifurcation diagram on the plane with coordinates $h$ and $k^2$. The diagram is presented in Fig.~11. \cite{bolsinov2010topology,orel1995rotation}.
It is symmetrical w.r.t. the line $k=0$ and have the following branches:
\begin{enumerate}
    \item $k = 0$, $h > -1$,
    \item $\displaystyle{h = \frac{3}{2}t^2 \pm 1}$, $2k = t^3$, $t \in \mathbb{R}$.
\end{enumerate}

The number of the Liouville tori changes as we cross the curve $\displaystyle{h = \frac{3}{2}t^2 + 1}$, $2k = t^3$: for large values of energy, there are two invariant tori corresponding to each point of the diagram. If $\displaystyle{h < \frac{3}{2}|2k|^{2/3} + 1}$, then there is one invariant torus.

\begin{proposition}
Let $h \ne 2k^2$, $k \ne 0$, $\displaystyle{h \ne \frac{3}{2}|2k|^{2/3} \pm 1}$ and the frequencies of motion on the torus are rationally independent. Then the line of nodes has the main motion $\Lambda = 0$.
\end{proposition}
\begin{proof}
If inequality $h < \frac{3}{2}|2k|^{2/3} + 1$ is satisfied, then for any given $k^2$ and $h$ we have one invariant torus (Fig.~11). Consider the projection of this torus onto the plane $(r, \gamma_3)$. If $(p,q,r,\gamma_1,\gamma_2,\gamma_3)$ is a solution of system (\ref{gor_ints}), then $(-p,-q,r,\gamma_1,\gamma_2,-\gamma_3)$ is also a solution. Therefore the projection is symmetric w.r.t. the line  $\gamma_3 = 0$. From Lemmas 11 and 12, we obtain that the main motion is zero.

If $h > \frac{3}{2}|2k|^{2/3} + 1$, then there are two invariant tori. Their projection also has the above symmetry. Let us now show that the projection of each torus is symmetric (Fig.~9) by proving that the points $(p,q,r,\gamma_1,\gamma_2,\gamma_3)$ and $(-p,-q,r,\gamma_1,\gamma_2,-\gamma_3)$ can be connected by a continuous path lying on the corresponding invariant torus.

For this we introduce the Chaplygin variables \cite{kharlamov1988}, which we denote by $x$ and $y$. All dynamics variables can be calculated by means of these variables:
\begin{equation}
\label{eq12}
    \begin{aligned}
            &p = \frac{1}{8}(XY_* + X_*Y), \quad q = \frac{1}{8}(X_*Y_* - XY),\\
            &r = x + y, \quad \gamma_1 = 1 - \frac{X_*^2 + Y_*^2}{2(x-y)},\\
            &\gamma_2 = \frac{1}{2(x-y)}(XX_* + YY_*),\\
            &\gamma_3 = \frac{1}{2(x-y)}(X_*Y - XY_*).
    \end{aligned}
\end{equation}
Here
\begin{equation}
    \begin{aligned}
            &X^2 = -Z(x), \quad X_*^2 = Z_*(x), \quad Y^2 = Z(y),\\
            &Y_*^2 = -Z_*(y), \quad Z(z) = z^3 - 2(h+1)z - 4k,\\
            &Z_*(z) = z^3 - 2(h-1)z - 4k.
    \end{aligned}
\end{equation}
Variables $x$ and $y$ are defined in the region where the inequalities $Z(x) \leqslant 0 \leqslant Z_*(x)$ and $Z_*(y) \leqslant 0 \leqslant Z(y)$ hold. It can be shown that if $h > \frac{3}{2}|2k|^{2/3} + 1$, then variables $x$ and $y$ belong to two non-intersecting intervals. Each interval can be presented as $[z_1^*, z_2]$ or $[z_1, z_2^*]$. Here $z_1$, $z_2$ are roots of the equations $Z = 0$, and, similarly, $z_1^*$, $z_2^*$ are roots of $Z_* = 0$. Note that the equality $x = y$ is never satisfied for $x$ and $y$ in the considered region.

Let $y$ be a fixed value and $x$ is changing in the above interval. Since the boundary points of this interval are roots of $Z$ and $Z_*$, then we can continuously change the value of $X$ to $-X$ and the value of $X_*$ to $-X_*$. From (\ref{eq12}) we have that the values of $r$, $\gamma_1$, $\gamma_2$ remain unchanged as we change $x$ periodically in the interval. At the same time, other variables change their signs.

\end{proof}
\begin{remark}
From the above result, it does not follow that the motion of the line of nodes is bounded. However, the amplitude of these oscillations can be majorized by any linear function of time.
\end{remark}

Proposition 5 is a generalization of a similar result  proved in \cite{kozlov1977qualitative} for the case of weak gravity.

\section*{Acknowledgement}
This work was supported by the Program of the Presidium of the Russian Academy of Sciences No 01 'Fundamental Mathematics and its Applications' under grant PRAS-18-01.
The author declares that he has no conflict of interest. 

\bibliographystyle{spmpsci}
\bibliography{sample}
\end{document}